\documentclass[12pt,a4paper]{amsart}

\usepackage{amssymb,amscd,latexsym,amsxtra}
\usepackage{enumerate}
\usepackage{dsfont}
\usepackage[scr=boondoxo,scrscaled=1.05]{mathalfa}
\usepackage{graphicx}
\usepackage{relsize}
\usepackage{framed}

\usepackage{amsfonts,bbold}
\usepackage[mathscr]{eucal}
\usepackage{calrsfs}
\usepackage{fge}

\newtheorem{thm}{Theorem}
\newtheorem{cor}[thm]{Corollary}
\newtheorem{lem}[thm]{Lemma}
\newtheorem{prop}[thm]{Proposition}

\theoremstyle{definition}
\newtheorem{defn}{Definition}

\newtheorem{prob}[thm]{Problem}

\newtheorem{rem}[thm]{Remark}

\newcommand{\sss}[1]{{\scriptscriptstyle #1}}

\newcommand{\um}{\mathcal U}

\newcommand{\pp}{\mathscr P}

\newcommand{\aaa}{\mathscr A}

\newcommand{\dd}{\mathscr D}

\newcommand{\ccc}{\mathscr C}
\newcommand{\uu}{\mathscr U}

\newcommand{\nn}{\mathscr N}
\newcommand{\ee}{\mathscr E}
\newcommand{\ff}{\mathscr F}
\newcommand{\kk}{\mathscr K}

\newcommand{\N}{\mathds N}
\newcommand{\R}{\mathds R}
\newcommand{\C}{\mathds C}

\newcommand{\Q}{\mathds Q}

\mathchardef\mhyphen="2D

\newcommand{\st}{\text{s.t.}}

\newcommand\restr[2]{{
  \left.\kern-\nulldelimiterspace 
  #1 
  \vphantom{\big|} 
  \right|_{#2} 
  }}

\newcommand{\mysetminus}{\mathbin{\fgebackslash}}

\DeclareMathOperator{\spn}{span}

\DeclareMathOperator{\range}{range\, }

\DeclareMathOperator{\diam}{diam}

\begin{document}

\title{On the commutant of $B(H)$ in its ultrapower}

\author{Emmanuel Chetcuti}
\address{
Emmanuel Chetcuti,
Department of Mathematics\\
Faculty of Science\\
University of Malta\\
Msida MSD 2080  Malta} \email {emanuel.chetcuti@um.edu.mt}

\author{Beatriz Zamora-Aviles}
\address{
Beatriz Zamora-Aviles,
Department of Mathematics \\
Faculty of Science \\
University of Malta \\
Msida MSD 2080, Malta} \email {beatriz.zamora@um.edu.mt}

\date{\today}
\begin{abstract}
Let $B(H)$ be the algebra of bounded linear operators on a separable infinite-dimesional Hilbert space $H$. We  study the commutant of $B(H)$ in its ultrapower.  We characterize the class of non-principal ultrafilters for which this commutant is non-trivial.  Additionally,  we extend the class of ultrafilters for which the commutant is trivial.
\end{abstract}
\maketitle

\section{Introduction}

Let $\alpha$ be a non-principal ultrafilter on $\N$ and let $A$ be a unital $C^*$-algebra. We define,
\begin{align*}
  \aaa_\infty & :=\left\{(x_n)\in\prod_n A:\sup_n\Vert x_n\Vert<\infty\right\}
\end{align*}
and
\begin{align*}
 \nn_\alpha:=\biggl\{(x_n)\in\aaa_\infty:\lim_\alpha\Vert x_n\Vert=0\biggr\}.
\end{align*}

The quotient of $\aaa_\infty$ by  $\nn_\alpha$ is a C*-algebra and it is called the \emph{ultrapower of $A$ w.r.t. $\alpha$}. We denote this C*- algebra by $\ccc_\alpha(A)$.  Note that if $\kk\subset\ccc_\alpha(A)$ is self-adjoint, then
\[\ff_\alpha(\kk):= \{\mathbf b\in\ccc_\alpha(A):\mathbf b\mathbf a-\mathbf a\mathbf b=\mathbf 0\ \forall \mathbf a\in\kk\}\]
is a C*-subalgebra of $\ccc_\alpha(A)$ and it is called the commutant of $\kk$ in $\ccc_\alpha(A)$.  It is easily seen that the algebra $A$ embeds isometrically in $\ccc_\alpha(A)$ which implies that $A$ and $\ff_\alpha(A)$ can be identified as C*-subalgebras of $\ccc_\alpha(A)$.  The invariant  $\ff_\alpha(A)$  is important for the classification of separable nuclear unital purely infinite simple $C^*$-algebras,  some of these applications can be found in \cite{Iz2004}, \cite{Ki}, \cite{KiPh2000} and \cite{Ph2000}.  If $A$ is separable and the Continuum Hypothesis is assumed then the isomorphism class of $\ff_\alpha(A)$ is independent of the choice of $\alpha$, see \cite{GeHa2001}. In the absence of the Continuum Hypothesis (CH), $\ff_\alpha(A)$ depends on the choice of $\alpha$ for every infinite-dimensional separable C*-algebra $A$,  see \cite{Fa2009} and \cite{FaHaSh2013}.

 Let $B(H)$ be the algebra of bounded operators on a separable infinite dimensional Hilbert space and let $K(H)$ be the ideal of compact operators on $B(H)$. For the Calkin algebra  $C(H)=B(H)/K(H)$ (which is non-separable)  Kirchberg proved in \cite{Ki2006} that $\ff_\alpha(C(H))=\C$ for any ultrafilter $\alpha$. This implies that  $\ff_\alpha(B(H))\subset \C + \ccc_\alpha(K(H))$.   Kirchberg asked whether  $\ff_\alpha(B(H))=\C$ for any ultrafilter $\alpha$.

It was surprising when I. Farah,  C. Philips and J. Steprans in \cite{FaPhSt2010} proved that under the Continuum Hypothesis, $\ff_\alpha(B(H))$ depends on the choice of the ultrafilter $\alpha$.  More specifically,  they proved that if $\alpha$ is a selective ultrafilter then $\ff_\alpha(B(H))=\C$.  They also constructed an ultrafilter in $ZFC$ for which $\ff_\alpha(B(H))\neq \C$  that they called \emph{flat}.   J. Steprans also proved that  $\ff_\alpha(B(H))= \C$ for any P-point ultrafilter. Recall that every selective ultrafilter is a P-point and neither the existence of selective nor P-point ultrafilters are provable from the axioms of ZFC;  Shelah showed that it is consistent that there are no P-point ultrafilters,  see \cite{Wi1982} .  However,  they exist assuming the Continuum Hypothesis,  \cite{Ma1977} and \cite{RuW1956}.

The authors of \cite{FaPhSt2010} posed the following two questions:
\begin{enumerate}[1.]
\item Is there an ultrafilter $\alpha$ on $\N$ in ZFC such that $\ff_\alpha(B(H))=\C$?
\item If an ultrafilter $\alpha$ satisfies $\ff_\alpha(B(H))\neq \C$, does if follow that $\alpha$ is flat?
\end{enumerate}
In this paper we give a positive answer to the second question.  In addition, we make other advancements with regard to the first question by introducing the class of \emph{quasi P-point ultrafilters} and show that for such ultrafilters the relative commutant of $B(H)$ is trivial.   In our analysis we also show that the relative commutant of $B(H)$ is  trivial for ultrafilters satisfying the 3f-property.  The following diagram summarizes these results.

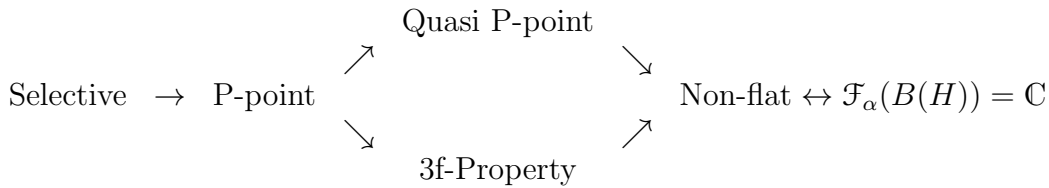
\[\begin{array}{ccccccc}
  { } & { } & { } &  & \text{Quasi P-point} & { } & { } \\
  { } & { } & { } & \nearrow & { } & \searrow & { }\\
  \text{Selective} & \rightarrow & \text{P-point} & { } & { } & { } & \text{Non-flat}\leftrightarrow \ff_\alpha(B(H))=\C \\
    { } & { } & { } & \searrow & { } & \nearrow & { }\\
  { } & { } & { } &  & \text{3f-Property} & { } & { }
\end{array}\]

It is known that the existence of ultrafilters with the 3f-Property (and therefore of selective ultrafilters and P-points) cannot be proved in ZFC.

In order to prove  $\ff_\alpha(B(H))=\C$ for every selective ultrafilter $\alpha$, the authors in \cite{FaPhSt2010} used a  result by Sherman \cite{She2010} about central sequences in $B(H)$.  They also used a non-trivial characterization of selective ultrafilters by Mathias in \cite{Ma1977}.  It is worthwhile to mention that our methods are different in that they are totally elementary.

This paper is organized as follows: In section \ref{sec2},  we set up notation and introduce some preliminaries used in the rest of the paper.  In section \ref{sec3} we recall the definitions of the classes of ultrafilters we consider.  Of particular interest is the characterization of \emph{P-point} ultrafilters in terms of a quasi order on the Baire space $\N^\N$ given in Corollary \ref{cPpoint}.   This characterization of P-points leads to the definition of a possibly larger class of ultrafilters which we call \emph{quasi P-point} ultrafilters, in Definition \ref{qppoint} and Proposition \ref{quasi}.  In addition we consider ultrafilters with the \emph{three functions property} (3f-property for short).  This property for ultrafilters was studied in  \cite{BlaDoRa2013} and it was  later isolated by A. Blass in  \cite{Bla2015}.  It should be noted that our definition is weaker.  It is easily seen that every P-point has the 3f-property,  but there are ultrafilters that have the 3f-property which are not P-points \cite{BlaDoRa2013}.  Hence this class of ultrafilters properly contains the P-point ultrafilters.  Some  properties of ultrafilters with the 3f-property are proven.  In particular,  we prove that no flat ultrafilter can satisfy the three function property,  Theorem \ref{f3f}.  Section 5 explores $\ff_\alpha(B(H))$ w.r.t. the ultrafilters introduced in Section \ref{sec3}.  We prove the converse of \cite[Theorem 4.1]{FaPhSt2010}, namely that if  $\ff_\alpha(B(H))\neq \C$, then $\alpha$ is flat.  This result combined with Corollary \ref{f3f},  implies that $\ff_\alpha(B(H))= \C$ for any ultrafilter with the 3f-property.  Further we show in Theorem \ref{qp} that if $\alpha$ is a quasi P-point then $\ff_\alpha(B(H))= \C$.
In section \ref{sec4},  we emphasize how \cite[Theorem 1.5]{FaPhSt2010}  can be stated and proved in a more general setting.  This result could  be of independent interest.

\section{Some notation and preliminary comments}\label{sec2}

\begin{enumerate}[{\rm (1)}]
\item Let $\nn$ denote the Baire space $\N^\N$ and $\nn^\uparrow$ the set of all strictly increasing members of $\nn$.
\item For every $k\in\N$ let $[k,\rightarrow]$ denote the set $\{n\in\N:n\ge k\}$.   For any infinite $A\subset \N$ let $\mathcal F(A)$ denote the cofinite filter in $A$.  We shall simply write $\mathcal F$ instead of $\mathcal F(\N)$.
\item Let $f\in\nn$ and $\alpha$ a filter on $\N$.
\begin{enumerate}[{\rm (i)}]
\item If P is a property associated with functions (like injectivity, surjectivity, etc.) we say that $f$ is $\alpha$-P (or has property P mod $\alpha$) when there exists $A\in\alpha$ s.t. $f|_A$ has property P.
\item The family $\alpha^f:=\{A\subset \N:f^{-1}A\in \alpha\}$ is a filter on $\N$ containing $\{f[A]:A\in\alpha\}$.  If $\alpha$ is a maximal filter, then $\alpha^f$ is also a maximal filter.
\item Note that $\alpha^f$ is principal if and only if $f$ is $\alpha$-constant, and if $\alpha$ is principal, then so is $\alpha^f$.  When $f$ is surjective, (i.e. $f[f^{-1}A]=A$ for every $A\subset \N$) then one can show that $\alpha^f=\{f[A]:A\in\alpha\}$.
\end{enumerate}
\item We shall primarily be concerned with non-principal ultrafilters on $\N$. Let $\beta\omega\mysetminus \omega$ denote the set of all non-principal ultrafilters on $\N$.  Unless otherwise stated: `ultrafilter on $\N$' means `non-principal ultrafilter on $\N$'.
\item  Let $\alpha\in\beta\omega\mysetminus \omega$.
\begin{itemize}
\item `$A_n\downarrow \emptyset$ in $\alpha$' means that $(A_n)$ is a  decreasing sequence in $\alpha$ and  $\bigcap_nA_n=\emptyset$.
\item It is easy to see that if $f=g\mod \alpha$ then $\alpha^f=\alpha^g$.  We recall the fact that if $\alpha^f=\alpha$ then $\{n\in\N:f(n)=n\}\in\alpha$, \cite[Theorem 3.3]{Booth1970}.
\item If $\beta\in\beta\omega\mysetminus \omega$ and there exists a bijection  $f\in\nn$ satisfying $\beta=\alpha^f$, we write $\alpha\sim_{\sss{{\rm RK}}}\beta$.\footnote{Note that $\alpha\sim_{\sss{{\rm RK}}}\beta$ if and only if there exists $\alpha$-injective $f\in\nn$ satisfying $\beta=\alpha^f$.}  We recall that $\sim_{\sss{{\rm RK}}}$ is an equivalence relation on $\beta\omega\mysetminus \omega$.
\item The Rudin-Keisler ordering on $\beta\omega\mysetminus \omega$ is defined as follows:    $\beta\leqslant_{\sss{{\rm RK}}}\alpha$ if $\beta=\alpha^f$ for some non $\alpha$-constant $f\in\nn$.  We recall that the Rudin-Keisler ordering induces a partial order on $(\beta\omega\mysetminus \omega)/\sim_{\sss{{\rm RK}}}$, and furthermore,  every selective ultrafilter is an atom in this partially ordered set.
\item Let $(X_n)$ be a sequence of subsets of $\N$ satisfying $\bigcap_n X_n=\emptyset$.   Let $A_n:=\bigcap_{i=1}^n X_i$ and let $f\in\nn$ be the function defined by $f(i):=\min\{n\in\N:i\notin A_{n+1}\}$.  Note that $A_n\downarrow \emptyset$ in the powerset of $\N$. Note further that for $k\ge 2$, $f^{-1}\{k\}=A_k\mysetminus A_{k+1}$.  $f$ is uniquely determined; we call it the  `indicator function associated with $(X_n)$'.
\end{itemize}
\end{enumerate}

\section{On some different ultrafilters}\label{sec3}

\subsection{Selective ultrafilters and P-points}

We recall that for an ultrafilter $\alpha\in \beta\omega\mysetminus \omega$, the following conditions are equivalent:
\begin{enumerate}[{\rm(i)}]
  \item For every partition $\{A_n:n\in\N\}$ of $\N$ with $A_n\notin\alpha$, there is $X\in\alpha$ such that $|X\cap A_n|=1$ for each $n\in\N$,
  \item Every $g\in\nn$ is either $\alpha$-injective or $\alpha$-constant,
  \item For every $E\subset [\N]^2$ there exists $A\in\alpha$ such that either $[A]^2\subset E$ or $[A]^2\cap E =\emptyset$,
  \item For every analytic\footnote{Recall that a subset of a Polish space is analytic if it is a continuous image of a Borel subset of a Polish space.} $E\subset [\N]^\infty$ there is $A\in\alpha$ such that either $[A]^\infty\subset E$ or $[A]^\infty\cap E=\emptyset$.
\end{enumerate}
In this case $\alpha$ is said to be \emph{selective} (or \emph{Ramsey}).   The equivalence of {\rm(i)-(iii)}, and the implication {\rm(iv)}$\Rightarrow$ {\rm(i)} are not difficult to see (see, for example, \cite[Theorem 1.3]{FaPhSt2010}) but the implication {\rm(i)}$\Rightarrow${\rm(iv)} is the Mathias's Theorem \cite{Ma1977}.  It is well known that selective ultrafilters  exist under $CH$ \cite{RuW1956}.

The ultrafilter $\alpha\in\beta\omega\mysetminus \omega$ is called a \emph{P-point} (or \emph{weakly selective}) if for every partition $\{A_n:n\in\N\}$ of $\N$ with $A_n\notin\alpha$, there exists $X\in\alpha$ such that $X\cap A_n$ is finite for each $n\in\N$.  Clearly every selective ultrafilter is a P-point.  However, the converse is not true.  It was proven by Mathias and others that the Continuum Hypothesis ensures the existence of  P-point ultrafilters that are non selective, (K. Kunen in \cite{Ku1976} proved that under MA there exist P-points that are not selective). On the other hand Shelah in \cite{Sh1982} constructed a model of ZFC in which there exists (up to isomorphism) exactly one P-point  and this P-point has to be selective.  Shelah also showed that it is consistent that there are no P-point ultrafilters \cite{Wi1982}. Therefore the existence of selective and P-point ultrafilters cannot be proved in ZFC.

For a non $\alpha$-constant function $f\in\nn$, we say that \emph{$f$ is $\alpha$-faithful} when there exists $K\in\alpha$ such that the condition $f[A\cap K]\in\alpha^f$ implies that $A\cap K\in\alpha$.  For $f,g\in\nn$ let us write $g\ll f$ if $g\le h f$ for some $h\in\nn^\uparrow$.

\begin{prop}\label{faith} Let $f\in\nn$ and $\alpha\in\beta\omega\mysetminus \omega$.  The following statements are equivalent:
\begin{enumerate}[{\rm(i)}]
  \item $f$ is $\alpha$-faithful.
  \item $\alpha\sim_{\sss{{\rm RK}}} \alpha^f$.
  \item There exists a bijection $g\in\nn$ such that $\alpha^{g^{-1}\circ f}=\alpha$.
  \item $f$ is $\alpha$-injective.
\end{enumerate}
\end{prop}
\begin{proof}
  {\rm(i)}$\Rightarrow${\rm(ii)}.  Let $K\in\alpha$ such that $f[A\cap K]\in\alpha^f$ implies $A\cap K\in\alpha$ for every $A\subset \N$.  Let $A\subset K$ such that $f[A]=f[K]$ and $f|_A$ is injective.  Faithfulness implies that
    $A\in\alpha$, and therefore $\alpha\sim_{\sss{{\rm RK}}} \alpha^f$.

  {\rm(ii)}$\Rightarrow${\rm(iii)}.   By definition, $\alpha\sim_{\sss{{\rm RK}}} \alpha^f$ implies that there exists $\alpha$-injective $g'\in\nn$ such that $\alpha^{g'}=\alpha^f$.  It is easy to see, then, that there exists a bijection $g\in\nn$ satisfying $g=g'\mod\alpha$ and $\alpha^g=\alpha^f$ and so, for an arbitrary $A\subset \N$, we have
  \[A\in\alpha^{g^{-1}\circ f}\ \Leftrightarrow\ f^{-1}( g (A))\in\alpha\ \Leftrightarrow\ g(A)\in \alpha^g\ \Leftrightarrow\ A\in \alpha.\]

  {\rm(iii)}$\Rightarrow${\rm(iv)}.   If there exists a bijection $g\in\nn$ such that $\alpha^{g^{-1}\circ f}=\alpha$, then (by the `fixed point theorem') $g^{-1}\circ f$ is $\alpha$-injective.  Therefore $f$ is $\alpha$-injective.

{\rm(iv)}$\Rightarrow${\rm(i)}.  If $K\in\alpha$, $\restr{f}{K}$ is injective, and $f[A\cap K]\in\alpha^f$ for some $A\subset \N$, then, $A\cap K=f^{-1}(f[A\cap K])\cap K\in\alpha$.  This shows that if $f$ is  $\alpha$-injective, then it is $\alpha$-faithful.
\end{proof}

\begin{lem}\label{l1}  For $f,g\in\nn$, we have $g\ll f$ if and only if there exists $s\in\nn^\uparrow$ such that
\[g^{-1}[s(n),\rightarrow]\subset f^{-1}[n,\rightarrow].\]
\end{lem}
\begin{proof}
  Suppose that $g\le h\circ f$ for some $h\in\nn^\uparrow$.  Then $g(i)\ge h(n)$ implies $h(f(i))\ge h(n)$ and therefore $f(i)\ge n$, i.e. $g^{-1}[h(n),\rightarrow]\subset f^{-1}[n,\rightarrow]$.  Conversely, suppose that there exists $s\in \nn^\uparrow$ such that $g^{-1}[s(n),\rightarrow]\subset f^{-1}[n,\rightarrow]$.  The function $h\in\nn$ defined by $h(i):=s(i+1)$ is strictly increasing. If $g(i)\ge h(f(i))$ for some $i\in\N$, then   $g(i)\ge s(f(i)+1)$, and therefore $f(i)\ge f(i)+1$.  So $g<h\circ f$.  This proves the lemma.
  \end{proof}

  If we set $f\sim g$ when there exists a bijective $h\in\nn$ satisfying $f=h g$, then $\sim$ defines an equivalence relation on $\nn$ and $\ll$ induces a partial order on $\nn/\sim$.

\begin{prop}\label{finitetoone} Let $\alpha\in\beta\omega\mysetminus \omega$ and $f\in\nn$.  The following statements are equivalent:
  \begin{enumerate}[{\rm(i)}]
    \item $f$ is $\alpha$-finite-to-one.
    \item $g\ll f$ for some $\alpha$-injective $g\in\nn$.
    \item There exists $\alpha$-injective $g\in\nn$ and $h\in\nn$ such that $g\le h\circ f$.
    \item There is  $g\in\nn$ that is $\alpha$-injective and $h\in\nn^\uparrow$ such that  $g^{-1}[h(n),\rightarrow]\subset f^{-1}[n,\rightarrow]$ for every $n\in\N$.
  \end{enumerate}
\end{prop}
\begin{proof}
The equivalence of {\rm(ii)} and {\rm(iv)} follows by Lemma \ref{l1}.   The implication {\rm(ii)} $\Rightarrow$ {\rm(iii)} is trivial.  Let us show {\rm(i)} $\Rightarrow$ {\rm(ii)}.  Let $X\in\alpha$ such that $f$ is finite-to-one on $X$.  For every $i\in\N$ let $X_i:=f^{-1}\{i\}\cap X$ and let $n_i:=|X_i|$.  Express the elements of $X_i$ in ascending order, i.e. $X_i=\{k^i_1,\dots,k^i_{n_i}\}$ and $k^i_s<k^i_t$ when $s<t$.  Define $h,g\in\nn$ by setting
\[h(i):=\sum_{j=1}^{i}n_j\,+\, i\qquad(i\in\N)\]
and
\[g(k):=\begin{cases}
              \sum_{j=1}^{i-1}n_j\,+\,i\,+\,p & \mbox{if there exist } i\in\N,\, 1\le p\le n^i \mbox{ such that } k=k^i_p \\
              h(f(k)), & \mbox{otherwise}.
            \end{cases}\]
            It is easy to verify that $h\in\nn^\uparrow$, $g$ is $\alpha$-injective and $g\le h f$.

Remains to be shown that {\rm(iii)} $\Rightarrow$ {\rm(i)}; if $g\in\nn$ is $\alpha$-injective and $h\in\nn$ satisfies $g\le h\circ f$, then, there exists $X\in\alpha$ such that $\restr{g}{X}$ is injective and, therefore, $\restr{f}{X}$ is finite-to-one.
\end{proof}

The following characterisation of P-points will be useful.

\begin{cor}\label{cPpoint}
The ultrafilter $\alpha\in\beta\omega\mysetminus \omega$ is a P-point if and only if for every non $\alpha$-constant $f\in\nn$ there exists $g\in\nn$ that is $\alpha$-faithful such that $g\ll f$.
\end{cor}

\subsection{Ultrafilters on a countable metric space.}

Let $X=\{x(n):n\in\N\}$ be a countable metric space. We suppose that all the $x(n)$'s are distinct, so that we need not distinguish between $X$ and $\N$.   A subset $A\subset \N$ is said to be \emph{$\delta$-discrete}, where $\delta>0$, if the distance between any two points of $A$ is at least $\delta$.  The ultrafilter $\alpha$ on $\N\ (=X)$ is \emph{discrete} if there exists $\delta>0$ and $A\in\alpha$ such that $A$ is $\delta$-discrete.  On the other extreme, $\alpha$ is said to be \emph{Cauchy} if for every $\varepsilon>0$ there exists $A\in\alpha$ satisfying $\diam A\le \varepsilon$.

\begin{prop}\label{P3}
Let $\alpha$ be an ultrafilter on an infinite countable metric space $X$.
\begin{enumerate}[{\rm(i)}]
\item If $\alpha$ is not Cauchy, there exists $\delta>0$ such that every $A\in\alpha$ contains an infinite $\delta$-discrete subset.
\item If $\alpha$ is selective, then $\alpha$ is either Cauchy or discrete.
\end{enumerate}
\end{prop}
\begin{proof}
If $\alpha$ is not Cauchy there exists $\delta>0$ such that $\diam A>2\delta$ for every $A\in\alpha$.  Let $x:\N\to X$ be an enumeration of $X$ and let $\rho$ denote the metric function.  For every $k\in\N$ define
  \[X_k:=\bigl\{n\in\N:\rho\bigl(x(n),x(i)\bigr)>\delta\quad\forall i\le k\bigr\}.\]
 It can readily be seen that $X_k\in\alpha$ for every $k\in\N$, and $\bigcap_k X_k=\emptyset$.
 Let $f$ be the indicator function associated with $(X_k)$.  Note that if $f(n)\ge m$, then $\rho\bigl(x(n),x(i)\bigr)>\delta$ for every $i\le m$.  We shall recursively construct a strictly increasing sequence $(n_k)$ in $\N$ such that $f(n_k)<n_{k+1}$.  Let $n_1:=1$.  The set $Y_1:=\{n\in\N:n>f(n_1)\}\in\alpha$.  Let $n_2\in Y_1$ be arbitrary.  Then $\rho\bigl(x(n_2),x(n_1)\bigr)>\delta$ and therefore $n_2>n_1$.  Suppose that $n_1<n_2<\cdots<n_k$ have been selected and $f(n_i)<n_{i+1}$ for every $i\le k-1$.  The set $Y_k:=\{n\in\N:n>f(n_k)\}\in\alpha$.  Let $n_{k+1}\in Y_k$.  Then $\rho\bigl(x(n_{k+1}),x(n)\bigr)>\delta$ for every $n\le n_k$, and in particular, $n_{k+1}>n_k$.  One has that either $O:=\bigcup_{i\in\N}[n_{2i-1},n_{2i}]$ or $E:=\bigcup_{i\in\N}[n_{2i},n_{2i+1}]$ belongs to $\alpha$.  Without loss of generality, suppose that $O\in\alpha$.  Observe that if $n\ge n_{2i-1}$ and $i>1$, then $n> f(n_{2i-2})$ and therefore $\rho\bigl(x(n),x(k)\bigr)>\delta$ for every $k< n_{2i-2}$.\\
Every $A\in\alpha$ must intersect an infinite number of the intervals $[n_{2i-1},n_{2i}]$, i.e. every $A\in\alpha$ contains an infinite $\delta$-discrete subset.  This proves {\rm(i)}.  If $\alpha$ is selective, there exists $A\in\alpha$ such that $A\cap [n_{2i-1},n_{2i}]$ consists precisely of one element, for every $i\ge 1$, i.e. $A$ is $\delta$-discrete.
\end{proof}

\begin{defn}\label{qppoint}
   An ultrafilter $\alpha\in \beta\omega\mysetminus \omega$ is called a \emph{quasi P-point} if for every countable and decreasing uniformity base $(\ee_k)$  on $\N$, one of the following assertions is true:
  \begin{enumerate}[{\rm(i)}]
    \item $\exists X\in\alpha\quad \exists k\in\N\quad\st\quad X^2\cap \ee_k\mysetminus \dd=\emptyset$, where $\dd$ equals the diagonal of $\N^2$,
    \item $\forall k\in\N\quad\exists X\in\alpha\quad\st\quad X^2\subset \ee_k$,
    \item $\forall A_n\downarrow\emptyset\ (\text{in }\alpha)\quad \forall k\in\N\quad \exists X\in\alpha\quad \st$
    \[X^k_n:=\{i\in X:\exists j\in A_n\ \st\ (i,j) \in\ee_k\}\]
    is cofinite in $X$, for every $n\in\N$.
  \end{enumerate}
\end{defn}

For every $\varepsilon>0$ we can define the relation $\prec_\varepsilon$ on $\nn$ by setting  $g\prec_\varepsilon f$ if  there exists $s\in\nn^\uparrow$ such that\footnote{For two subsets $A$ and $B$ of a metric space let us write $A\subset_\varepsilon B$ when $d(a, B)<\varepsilon $ for every $a\in\ A$.}
\[g^{-1}[s(n),\rightarrow]\subset_\varepsilon f^{-1}[n,\rightarrow]\qquad(\forall\ n\in\N).\]
The following proposition characterizes quasi P-points in a similar way that Corollary \ref{cPpoint} characterizes P-points.  This will be useful in the proof of Theorem \ref{qp}.

\begin{prop}\label{quasi}
  An ultrafilter $\alpha\in\beta\omega\mysetminus \omega$ is a quasi P-point if and only if for every metric on $\N$, w.r.t. which $\alpha$ is neither discrete nor Cauchy, the following statement holds
    \[\forall f\in\nn\quad \forall\varepsilon>0\quad \exists\,\alpha\text{-faithful}\ g\in\nn\ \st\  g\prec_\varepsilon f.\]
\end{prop}
\begin{proof}
$(\Rightarrow )$  Let $\rho$ be an arbitrary metric on $\N$, and let $f\in\nn$ and $\varepsilon>0$ be given.  The sets
\[\ee_k:=\{(m,n)\in\N^2:\rho(m,n)<1/k\}\qquad (k\in\N)\]
is a countable and  decreasing uniformity base for $\N$.  For every $n\in\N$ the set $A_n:=f^{-1}[n,\rightarrow]$ belongs to $\alpha$ and $\bigcap_n A_n=\emptyset$.  Let $k>1/\varepsilon$.  By hypothesis, $\alpha$ is neither discrete nor Cauchy, so there exists $X\in\alpha$ such that
 \[X^k_n:=\{i\in X:\exists j\in A_n\ \st\ (i,j) \in\ee_k\}\]
    is cofinite in $X$, for every $n\in\N$, i.e. for every $n\in\N$ there exists a subset $X^k_n$ that is cofinite in $X$, and satisfying $X^k_n\subset_\varepsilon A_n=f^{-1}[n,\rightarrow]$.   Let $F_1$ be a finite subset of $X$ such that $X\mysetminus F_1\subset_{\varepsilon} f^{-1}[2,\rightarrow]$, and for every $n\ge 2$ let $F_n\subset X\mysetminus \bigcup_{i=1}^{n-1}F_i$ such that $X\mysetminus \bigcup_{i=1}^n F_i\subset_\varepsilon f^{-1}[n,\rightarrow]$.  Let $h\in\nn$ be the function defined by $h[F_n]=\{n\}$ and $h[\N\mysetminus X]=\{1\}$.  Then, $h$ is $\alpha$-finite-to-one and $h^{-1}[n,\rightarrow]\subset_{\varepsilon} f^{-1}[n,\rightarrow]$, for every $n\in\N$.  By virtue  of Propositions \ref{finitetoone} and \ref{faith}, it follows that there exists an $\alpha$-faithful $g\in\nn$ such that $g\prec_\varepsilon f$.\\
$(\Leftarrow)$ Let $(\ee_k)$ be a countable and decreasing uniformity base on $\N$.  We recall (see, for example, \cite[pg. 434, Theorem 8.1.21]{Engelking}) that in this case the uniformity on $\N$ generated by $(\ee_k)$ is induced by some metric $\rho$.  Note that for every $k\in \N$ there exists $\varepsilon_k>0$ such that $\rho(n,m)<\varepsilon_k$ implies that $(n,m)\in \ee_k$.  Suppose that {\rm(i)} and {\rm(ii)} of Definition \ref{qppoint} fail.  Then, $\alpha$ is neither discrete, nor Cauchy, w.r.t. $\rho$.  Suppose that $A_n\downarrow \emptyset$ in $\alpha$.  Let $f$ denote the indicator function associated to $(A_n)$.  Then, for every $k\in\N$ there exists $\alpha-$faithful $g\in\nn$ satisfying $g\prec_{\varepsilon_k} f$, i.e. (by virtue of Propositions  \ref{finitetoone} and \ref{faith}) there exists an $\alpha$-finite-to-one function $h\in\nn$ satisfying
\[h^{-1}[n,\rightarrow]\subset_{\varepsilon_k} f^{-1}[n,\rightarrow],\]
for every $n\in\N$.  Let $X\in\alpha$ such that $\restr{h}{X}$ is finite-to-one.  Fix $n\in\N$.  If $i\in X$ and $h(i)\ge n$, then there exists $j\in\N$ such that $\rho(i,j)<\varepsilon_k$ and $f(j)\ge n$.  This implies (since $\restr{h}{X}$  is finite-to-one) that the set $\{i\in X:\exists j\in A_n\ \st\ \rho(i,j)<\varepsilon_k\}$ is cofinite in $X$, and since
\[\{i\in X:\exists j\in A_n\ \st\ \rho(i,j)<\varepsilon_k\}\ \subset\ \{i\in X:\exists j\in A_n\ \st\ (i,j)\in\ee_k\}\]
it follows that $\{i\in X:\exists j\in A_n\ \st\ (i,j)\in\ee_k\}$  is cofinite in $X$.
\end{proof}

\begin{cor}
Every P-point is a quasi P-point.
\end{cor}
\begin{proof}
  This follows  by  Proposition \ref{quasi} and Corollary \ref{cPpoint}.
\end{proof}

The following lemma will be used in the proof of Theorem \ref{qp}.

\begin{lem}
Let $X_n\downarrow \emptyset$ in $\alpha$, where $\alpha\in \beta\omega\mysetminus \omega$ and let $f$ be the indicator function of $(X_n)$.
For every  $g\in\nn$ satisfying $g\prec_\varepsilon f$, and for every $h\in\nn^\uparrow$, there exists  a family of consecutive intervals $\{I_n:n\in\N\}$  such that  $\bigcup_{n}I_n\in\alpha^g$ and
\[g^{-1} I_{n+1}\subset_\varepsilon X_{\sss{{h(\max I_n)}}}\qquad(\forall\ n\in\N).\]
\end{lem}
\begin{proof}
Let $s\in\nn^\uparrow$ such that $g^{-1}[s(n),\rightarrow]\subset_\varepsilon f^{-1}[n,\rightarrow]$ for every $n\in\N$.  Thus, for every $i\in\N$ satisfying $g(i)\ge s(h(n))$, there exists $i'\in\N$ such that $\rho(i,i')<\varepsilon$ and $f(i')\ge h(n)$.  Therefore $\rho(i,X_{h(n)})<\varepsilon$.  Let $n_1:=1$, and for $k>1$, let $n_k:=s(h(n_{k-1}))$.
 Then, either $X:=\bigcup_{k=1}^\infty [n_{2k},n_{2k+1}]$ or $Y:=\bigcup_{k=1}^\infty]n_{2k-1},n_{2k}[$ belongs to $\alpha^g$.  Without loss of generality, assume that $X\in \alpha^g$.  Let $I_k:=[n_{2k},n_{2k+1}]$.  If $i\in \N$ and $g(i)\in I_{k+1}$, then $g(i)\ge s(h(n_{2k+1}))$, and therefore $\rho(i,X_{\sss{h(n_{2k+1})}})<\varepsilon$.  This implies that  $g^{-1} I_{k+1}\subset_\varepsilon X_{\sss{h(\max I_{k})}}$.
\end{proof}

\subsection{Flat Ultrafilters}

Let  $x=(\eta_i)\in\ell^\infty$.  For every $f\in \nn^\uparrow$ and $k\in\N$,  define
\[\omega^k_f(x):=\max\{|\eta_i-\eta_j|\,:\,f^{k-1}(1)\le i\le j< f^{k}(1)\},\]
 where $f^0(1):=1$ and $f^{k+1}(1):=f(f^{k}(1))$.  Let
  \[\omega_f(x):=\sup_{k\in\N}\omega^k_f(x).\]
In this way the set $\ell^\infty$ can be identified with a subset of $\R^\nn$ as follows:  To every $x\in\ell^\infty$ one associates $x^\ast:\nn\to \R^+$ defined by $x^\ast:f\mapsto \omega_f(x)$.

\begin{defn} A \emph{flatness scale for an ultrafilter $\alpha$} is a sequence $(s_n)$ in $[0,1]^\N$ satisfying:
\begin{enumerate}[{\rm(i)}]
  \item there exists $r>0$  such that $\sup_i s_n(i)\ge r$ and $\lim_{i\to\infty} s_n(i)=0$ for every $n\in \N$,
  \item $\lim_{n\to\alpha} \omega_f(s_n)=0$ for every $f\in\nn^\uparrow$.
\end{enumerate}
When $\alpha\in\beta\omega\mysetminus \omega$ admits a flatness scale, $\alpha$  is called a \emph{flat ultrafilter}.
\end{defn}

\begin{rem}\label{r1}Let $(a_{ij})$ be an infinite matrix with values in $[0,1]$.  Define $b_{ij}:=\sup_{k\ge j}a_{ik}$. For $p\le q$ in $\N$ we have
  \[\max\{|b_{ij}-b_{ik}|:p\le j\le k\le q\}\ =\ b_{ip}-b_{iq}\ =\ \max_{p\le j<q} a_{ij}\,\vee\, b_{iq}\ -\ b_{iq}\ \le\ \max_{p\le j<q} a_{ij}\ -\ a_{iq}\]
  and therefore
  \[\max\{|b_{ij}-b_{ik}|:p\le j\le k\le q\}\ \le\ \max\{|a_{ij}-a_{ik}|:p\le j\le k\le q\}.\]
   This shows that a flat ultrafilter admits a flatness scale $(s_n)$ satisfying  $s_n(i+1)\le s_n(i)\le s_n(1)=1$ for every $i$ and $n$ in $\N$.
\end{rem}

\begin{prop}\label{flatP1}
An ultrafilter $\alpha\in\beta\omega\mysetminus \omega$ is flat if and only if  there is a sequence of decreasing functions $(s_n)$ in $ [0,1]^\N$, such that
\begin{enumerate}[(i)]
\item $\lim_{i\to\infty}s_n(i)=0$, for all $n\in\N$.
\item $s_n(1)=1$, for all $n\in\N$.
\item $\lim_{n\to\alpha} \Vert s_n- s_n\circ f\Vert_{\sss{\infty}} =0$, for all $f\in \nn^\uparrow$.
\end{enumerate}
\end{prop}
\begin{proof}

Let $(t_n)$ be a flatness scale for $\alpha$ and let $r>0$  such that $\Vert t_n\Vert_{\sss{\infty}}=\sup_i t_n(i)\ge r$ for each $n\in\N$.  Define $s_n(j):=\sup_{k\ge j}t_n(k)/ \Vert t_n\Vert_{\sss{\infty}} $.  It is clear that $(s_n)$ satisfies (i) and (ii) and each $s_n$ is decreasing.  To check (iii), let $\varepsilon>0$ be given.  Then $\{n\in\N:\Vert t_n- t_n\circ f\Vert_{\sss{\infty}}<r\varepsilon\}\in\alpha$.  Hence by Remark \ref{r1},  $\{n\in\N:  \Vert t_n\Vert_{\sss{\infty}}\Vert s_n- s_n\circ f\Vert_{\sss{\infty}}<r\varepsilon\}\in\alpha$ and this implies $\{n\in\N:  \Vert s_n- s_n\circ f\Vert_{\sss{\infty}}<\varepsilon\}\in\alpha$.

\end{proof}

Note that this characterization was the original definition of flat ultrafilter  \cite[Definition 3.2]{FaPhSt2010}.

\begin{prop}\cite[Theorem 3.3]{FaPhSt2010}
Flat ultrafilters exist in ZFC.
\end{prop}
\begin{proof}
 Let $X$ denote the set of all decreasing sequences in $[0,1]\cap \Q$, starting at $1$, and eventually equal to $0$.  Since $X$ is countable, the existence of a flat ultrafilter would follow if we  show that $0^\ast$ belongs to the closure of $X^\ast:=\{x^\ast:x\in X\}$.   Let $\uu\subset \R^\nn$ be an open neighborhood of $0^\ast$, and let  $\{f_1,\dots,f_n\}\subset\nn$ and $\varepsilon>0$ be such that
\[\{x^\ast\ :\ x\in X,\,\omega_{f_i}(x)<\varepsilon\ \forall i\le n\}\ \subset\ \uu.\]
It suffices to exhibit $x\in X$ satisfying $\omega_f(x)<\varepsilon$, where  $f:=\max(f_1,\dots,f_n)$.  To this end,  define, recursively, the sequence $(k_i)$ by setting $k_1:=1$ and, for $i\ge 2$, $k_i:=f^{i-1}(1)$.
Choose $m\in\N$ such that $1/m<\varepsilon$ and let $x:=(\xi_i)$ be defined by
\[\xi_i:=\begin{cases}
0\quad&\text{if }i\ge k_{m+1}\\
1-(p-1)/m\quad&\text{if }  k_{p}\le i<k_{p+1}\text{ for some }p\le m.
\end{cases}\]
It is clear that $\omega_f(x)=1/m$.
\end{proof}

\begin{prop} \label{3f-lem3} Let $(s_n)$ be a flatness scale for $\alpha\in\beta\omega\mysetminus \omega$.  Then the following statements are true.
\begin{enumerate}[{\rm(i)}]
\item $\lim_{n\to\alpha} s_n(i)\,=\,\lim_{n\to\alpha} s_n(j)$ for every $i,j\in\N$.
\item  For every $r>0$ and for every $A\in\alpha$, the set
\[K(r,A):=\bigl\{k\in\N:\bigl(s_n(k):n\in A\bigr)\text{ {\rm is eventually $\ge r$}}\bigr\}\]
is finite.
\end{enumerate}
\end{prop}
\begin{proof}
{\rm(i)}~Define $s(i):=\lim_{n\to\alpha}s_n(i)$ and suppose that $|s(i)-s(j)|>\varepsilon$ for some $\varepsilon>0$ and $i,j\in\N$.  Let $A\in\alpha$ such that $|s_n(i)-s(i)|$ and $|s_n(j)-s(j)|$ are both less than $\varepsilon/4$, for every $n\in A$.  This implies that $|s_n(i)-s_n(j)|>\varepsilon/2$ for every $n\in A$.  This contradicts the fact that $(s_n)$ is a flatness scale for $\alpha$.\\
{\rm(ii)}~Suppose that $K(r,A)$ is infinite, for some $r>0$ and $A\in\alpha$.  We show that $(s_n)$ is not a flatness scale for $\alpha$.  To this end we shall construct a strictly increasing sequence $(\ell_i:i\ge0)$ such that
\[\bigcup_{i\ge 0}\bigl\{n\in A\ :\  |s_n(\ell_i)-s_n(\ell_{i+1})|\ge r/2 \bigr\}\in \alpha.\]
Let $\{k_n:n\in\N\}$ be an enumeration of $K(r,A)$ such that  $k_n<k_{m}$ for $n<m$.  For every finite subset $F$ of $A$ let
\[m(F):=\min\{i\in K(r,A)\ :\ (\forall n\in F)\ (\forall j\ge i)\ (s_n(j)\le r/2)\}.\]
Let $\bigl(g(k_n):n\in\N\bigr)$ be a strictly increasing sequence in $A$ such that  $s_i(k)\ge r$ for every $i\in A$ satisfying $i\ge g(k_n)$.  We shall construct a strictly increasing sequence $(I_i:i\ge 0)$ of intervals of $A$ such that $\bigcup_{i\ge 0}I_i=A$, and
\begin{equation}\label{e1}
s_n(m(I_i))\ge r\qquad\text{and}\qquad s_n(m(I_{i+2}))\le r/2,
\end{equation}
for every $n\in I_{i+2}\mysetminus I_{i+1}$ and $i\ge 0$.    Note that we can suppose that $1\in A$.   Set $I_0:=\{1\}$ and,  recursively, let $I_{i+1}:=\{n\in A:1\le n\le g(m(I_i))\}$.  Observe that
\[s_n(m(I_i))\le r/2\quad\forall\,n\le g(m(I_{i-1}))\quad\text{and}\quad s_n(m(I_i))\ge r\quad\forall\,n\ge g(m(I_i))\,,\]
i.e.  (\ref{e1}) follows and $m(I_{i})>m(I_{i-1})$ for every $i\ge 1$.  Therefore,  $(I_i:i\ge 0)$ is  a strictly  increasing sequence of intervals of $A$ and so  $\bigcup_{i\ge 0}I_i=A$.

We conclude the proof as follows.  For $i\ge 0$ let $\ell_i:=m(I_{2i})$.  Then,
\[|s_n(\ell_i)-s_n(\ell_{i+1})|\ge r/2\]
for every $n\in I_{2i+2}\mysetminus I_{2i+1}$ by (\ref{e1}).  If $\bigcup_{i\ge 0}I_{2i+2}\mysetminus I_{2i+1}\in \alpha$ we are done.  Otherwise, we have that $\bigcup_{i\ge 1}I_{2i+1}\mysetminus I_{2i} \in\alpha$.  In this case, for every $i\ge 1$ we define $\ell_i:=m(I_{2i-1})$ and observe that $|s_n(\ell_i)-s_n(\ell_{i+1})|\ge r/2$ for every $n\in I_{2i+1}\mysetminus I_{2i}$ by (\ref{e1}).
\end{proof}

\subsection{$3$f-property}

\begin{defn}
The ultrafilter $\alpha\in\beta\omega\mysetminus \omega$ is said to have the \emph{three functions property}  (3f-property in short) if there exists a bijection $\phi:\N\to\N^2$,  such that for every $f\in\nn$, there exists $A\in\alpha$ satisfying either
\begin{enumerate}[{\rm(i)}]
\item $\restr{f}{A}$ is constant, or
\item  $\restr{f}{A}$ is finite-to-one, or
\item there is a one-to-one function $\overline f\in\nn$ such that $\restr{f}{A}=\overline f\circ\pi\circ
\restr{\phi}{A}$,  where $\pi:\N^2\to \N$ denotes the projection onto the first coordinate.
\end{enumerate}
\end{defn}

We will say in this case that $\alpha$ has the 3f-property w.r.t. $\phi$.  This property has been isolated by A.~ Blass in  \cite{Bla2015}.  It should be noted that the  definition given in \cite{Bla2015}  differs slightly from ours  because for the second case A.~Blass requires $f$ to be $\alpha$-injective instead of $\alpha$-finite-to-one.   It is easy to see that every P-point has the three function property (we can think of P-point ultrafilters as ultrafilters having the two functions property).  However there are ultrafilters with the three function property that are not P-points.  In fact, in \cite{BlaDoRa2013}  the authors study certain ultrafilters added generically by a countably closed Forcing notion. They refer to these ultrafilters as "the next best thing to a P-point" ultrafilters. These ultrafilter enjoy  the 3f-property and  they are not P-points but they are weak P-points. They also consider sums of non-isomorphic selective ultrafilters indexed by a selective ultrafilter, they showed that these ultrafilters  share the 3f-property but they are neither P-points nor weak P-points.   In  \cite{Bla2015}, it is also shown that ultrafilters that are not P-points but satisfy certain Ramsey-like partition properties have the 3f-property.  Hence the class of ultrafilters with the 3f-property properly contains the P-point ultrafilters.  However,  the existence of ultrafilters with the 3f-property cannot be proved within ZFC,  for if $\alpha$ has the 3f-property and it is not a P-point, then there is $f:\N\to \N$ witnessing the later and it is easy to see that $\alpha^f=\{f^{-1}(A): A\in\alpha\}$ is a P-point.

For a subset $X\subset \N^2$ and $i\in\N$ we set  $X^i:=\{ j\in \N: (i,j)\in X\}$. Recall that if $\phi$ is bijective,  then $\alpha^\phi=\{\phi(A): A\in \alpha\}$ and $\alpha\thicksim_{RK}\alpha^\phi$.

\begin{rem}  Suppose that the ultrafilter $\alpha$ satisfies the $3$f-property w.r.t. $\phi$. For every $A\subset \N$ let $\hat A:=\phi[A]$,
\[\hat{A}_0:=\bigcup\{\hat{A}\,^j\,:\,|\hat{A}\,^j|<\infty\},\]
and
\[\hat{A}_\infty:=\bigcup\{\hat{A}\,^j\,:\,|\hat{A}\,^j|=\infty\}.\]
Observe that if $\hat{A}_0\in\alpha^\phi$ for some $A\in\alpha$, then $\alpha$ is a P-point.  Otherwise, for every $A\in\alpha$, there exists $B\in\alpha$, $B\subset A$ such that every vertical section $\phi[B]\,^j$ $(j\in\N)$ is an infinite set.
\end{rem}

\begin{prop}\label{3fP1}
Let $(a_{i\,j}:i,j\in\N)$ be an infinite matrix with entries in $[0,1]$ satisfying:
\begin{enumerate}[{\rm(i)}]
\item $1=a_{i\,1}\ge a_{i\,j}\ge a_{i\,j+1}$, for every $i,j\in\N$, and
\item $\lim_{j\to\infty} a_{i\,j}=0$ for every $i\in\N$.
\end{enumerate}
Let $\alpha$ be an ultrafilter satisfying the $3$f-property.  Then, for every $0<\delta<1/4$ there exists  $A\in \alpha$ and two  non-decreasing sequences $(m_k)$ and $(n_k)$ in $\N$ satisfying $m_k<n_k$ for every $k\in\N$,  such that
\[A=\bigcup_{k\in\N}\{i\in A:a_{i\,m_k}- a_{i\,n_k} >\delta\}.\]
\end{prop}
\begin{proof}
Let $0<\delta<1/4$.  For every $i\in\N$ let $f(i):=\min\{j\in\N:a_{i\,j}<1-\delta\}$, $g(i):=\min\{j\in\N:a_{i\,j}<1-2\delta\}$ and $h(i):=\min\{j\in\N:a_{i\,j}<1-3\delta\}$.  Note that $f\le g\le h$.

If there exists $A\in\alpha$ and $n\in\N$ such that $f(i)=n$ for every $i\in A$ (i.e. $f$  is $\alpha$-constant) we let  $m_k:=1$ and $n_k:=n$ for every $k\in\N$.  Note that if $h$ is $\alpha$-constant, then so is $g$, and if $g$ is $\alpha$-constant, then so is $f$.

Suppose that $f$ is $\alpha$-finite-to-one.  In this case there exists $A\in\alpha$ and a partition $\{F_k:k\in\N\}$ of $A$ into finite subsets such that $f(i)=f(i')$ if and only if $i$ and  $i'$ belong to the same $F_k$ for some $k\in\N$.  Let $m_k\in\N$ such that $f[F_k]=\{m_k+1\}$.  After reorganization (if necessary) we can suppose that $(m_k)$ is (strictly) increasing.  Inductively, one can construct a (strictly) increasing sequence $(n_k)$ such that $g(i)\le n_k$ for every $i\in F_k$ and  $k\in\N$.  Observe that $a_{i\,m_k}\ge 1-\delta$ and $a_{i\,n_k}<1-2\delta$ for every $i\in F_k$.  A similar argument (involving $g$ and $h$)  holds if we suppose that $g$ is $\alpha$-finite-to-one.

Let us suppose that neither $f$ nor $g$ is $\alpha$-constant or $\alpha$-finite-to-one.  Then there are finite-to-one functions $\bar f$ and $\bar g$ in $\nn$, and $A\in\alpha$, such that $\restr{f}{A}=\restr{\bar f\circ\pi\circ \phi}{A}$ and $\restr{g}{A}=\restr{\bar{g}\circ\pi\circ\phi}{A}$.  Let $\{F_k:k\in\N\}$ be a partition of $\N$ into finite subsets such that $f(i)=f(i')$ if and only if $i$ and  $i'$ belong to the same $F_k$ for some $k\in\N$.  Let $m_k\in\N$ such that $f[F_k]=\{m_k+1\}$.   After reorganization, we can suppose that $(m_k)$ is (strictly) increasing.   Inductively, one can construct a (strictly) increasing sequence $(n_k)$ such that $\bar g(i)\le n_k$ for every $i\in F_k$ and  $k\in\N$.  Observe that if $i\in A\,\cap\,((\phi^{-1}\circ\pi^{-1}) (F_k))$, then $g(i)=(\bar g\circ\pi\circ\phi)(i)\le n_k$, and therefore, $a_{i\,m_k}\ge 1-\delta$ and $a_{i\,n_k}<1-2\delta$.
\end{proof}

\begin{cor}\label{f3f}
If an ultrafilter $\alpha$ is flat, then it cannot satisfy the $3$f-property.
\end{cor}
\begin{proof}
This follows by Propositions \ref{3fP1} and \ref{flatP1}; one simply needs to consider a function $f\in\nn^\uparrow$ satisfying $f(m_k)\ge n_k$ for every $k\in\N$.
\end{proof}

\begin{rem}\label{rem2}
  With the same notation of the proof of Proposition \ref{3fP1}, note that if we don't have the  (simple) case where $(m_k)$ and $(n_k)$ are constant (i.e. if $f$ is not $\alpha$-constant), we can define a strictly increasing sequence  $(k_j)$ in $\N$ such that $n_{k_j}<m_{k_{j+1}}$ for every $j\in\N$.  Observe then, that either
\[A_o=\bigcup_{\substack{j\in\N\\j\text{ is odd}}}\{i\in A:a_{i\,m_{k_j}}- a_{i\,m_{k_{j+2}}} >\delta\}\in\alpha\]
or
\[A_e=\bigcup_{\substack{j\in\N\\j\text{ is even}}}\{i\in A:a_{i\,m_{k_j}}- a_{i\,m_{k_{j+2}}} >\delta\}\in\alpha .\]
\end{rem}

\section{Ultrapowers of a metric space - some considerations}\label{sec4}

Let $(M,\rho)$ be a metric space  and let $\mathcal M$ denote the set of all bounded sequences in $M$.  Denote the members of $\mathcal M$ by $\mathbf x=(x_n)$, $\mathbf y=(y_n)$ and so on.  The following is easily seen to be an equivalence relation on $\mathcal M$
  \[\Delta:=\{(\mathbf x,\mathbf y)\ :\ \{n\in\N:\rho(x_n,y_n)<\varepsilon\}\in\mathcal F\ \forall \varepsilon>0\},\]
  and associated to the ultrafilter $\alpha\in\beta\omega\mysetminus \omega$ we further have the following equivalence relations:
  \begin{align*}
  \Delta^\alpha_0:= & \{(\mathbf x,\mathbf y)\ :\ \{n\in\N:x_n=y_n\}\in\alpha\}, \\
  \Delta^\alpha_1:= & \{(\mathbf x,\mathbf y)\ :\ \exists A\in\alpha\ \st\  \{n\in\N:\rho(x_n,y_n)<\varepsilon\}\in\mathcal F(A)\quad \forall \varepsilon>0 \},\\
    \Delta^\alpha_2:= &\{(\mathbf x,\mathbf y)\ :\ \{n\in\N:\rho(x_n,y_n)<\varepsilon\}\in\alpha\ \forall \varepsilon>0\}.
     \end{align*}
The functions
\[\rho_\infty:\mathcal M\times\mathcal M\to [0,1]:\bigl((x_n),(y_n)\bigr)\mapsto \sup_{n\in\N}\rho(x_n,y_n),\]
and
\[\rho_\alpha:\mathcal M\times\mathcal M\to [0,1]:\bigl((x_n),(y_n)\bigr)\mapsto \lim_\alpha\rho(x_n,y_n),\]
define,  respectively, a metric and a pseudo-metric on $\mathcal M$.

The following inclusions can easily be verified.
 \[\begin{array}{ccccc}
          {} & {} & {\Delta} & {} & {} \\
           {} & {} & \rotatebox{90}{$\supset$} & {} & {} \\
          {\Delta^\alpha_0} & {\subset} & \Delta^\alpha_1 & \subset & \Delta^\alpha_2
        \end{array}\]

In the complete lattice of all the  equivalence relations on $\mathcal M$ (ordered by set-theoretic inclusion) one has the equality $\Delta^\alpha_1=\Delta\vee\Delta^\alpha_0$.

\begin{prop}
  The closure of $\Delta^\alpha_0$ is equal to $\Delta^\alpha_2$ w.r.t. both  metrics $\rho_\infty\times\rho_\infty$ and $\rho_\alpha\times\rho_\alpha$.
\end{prop}
\begin{proof}
It is easy to see that $\Delta_0^\alpha$ is dense in $\Delta_2^\alpha$ w.r.t. $\rho_\infty\times\rho_\infty$ (and therefore w.r.t. $\rho_\alpha\times\rho_\alpha$).  We show that $\Delta^\alpha_2$ is closed w.r.t. $\rho_\alpha\times\rho_\alpha$.  (This will imply that $\Delta^\alpha_2$ is closed w.r.t. $\rho_\infty\times\rho_\infty$.)  Suppose that $(\mathbf x^n)$, $(\mathbf y^n)$ are sequences in $\mathcal M$ such that $\lim_{n\to\infty}\mathbf x^n=\mathbf x$,  $\lim_{n\to\infty} \mathbf y^n=\mathbf y$ (w.r.t. $\rho_\alpha$) and suppose that $(\mathbf x^n,\mathbf y^n)\in \Delta^\alpha_2$ for every $n\in\N$.  Fix an arbitrary $\varepsilon>0$.  We show that $\{i\in\N:\rho(x_i,y_i)<\varepsilon\}\in\alpha$.  Let $\delta:=\varepsilon/3$.  The assumption $(\mathbf x^n,\mathbf y^n)\in\Delta^\alpha_2$ implies that $\{i\in\N:\rho(x^n_i,y^n_i)<\delta\}\in\alpha$.  We also have that
    \[J:=\{n\in\N:\rho_\alpha(\mathbf x^n,\mathbf x)\vee\rho_\alpha(\mathbf y^n,\mathbf y)<\delta\}\in\mathcal F,\]
    and so, for every $n\in J$ we have
    \[\{i\in\N:\rho(x^n_i,x_i)\vee\rho(x^n_i,y^n_i)\vee\rho(y^n_i,y_i)<\delta\}\in\alpha.\]
    Hence, $\{i\in\N:\rho(x_i,y_i)<\varepsilon\}\in\alpha$.
\end{proof}

\begin{prop}\label{Ppoint1}
 If $\alpha$ is a P-point, then $\Delta^\alpha_1=\Delta^\alpha_2$.
 \end{prop}
 \begin{proof}  Suppose that $(\mathbf x,\mathbf y)\in\Delta^\alpha_2$.  Let $\varepsilon_i\downarrow 0$ in $\R$ and let $A_i:=\{j\in\N:\rho(x_j,y_j)<\varepsilon_i\}$.  Then $A_i\in\alpha$ and $A_i\downarrow$.  If $\bigcap_{i\in\N}A_i\in\alpha$, then $\{i\in\N:x_i=y_i\}\in\alpha$ and $(\mathbf x,\mathbf y)\in\Delta^\alpha_0$.  So, suppose that $\bigcap_{i\in\N}A_i\notin\alpha$ and let $B_j:=A_j\mysetminus \bigcap_{i\in\N}A_i$.  Then $B_j\downarrow \emptyset$ in $\alpha$ and -- being a P-point -- $\alpha$ contains $X$ such that $X\cap(B_j\mysetminus B_{j+1})$ is finite, for every $j\in\N$.  This implies that
 \[\lim_{\substack{j\in X\\j\to\infty}}\rho(x_j,y_j)=0,\]
 i.e. $(\mathbf x,\mathbf y)\in\Delta^\alpha_1$.
 \end{proof}

 \begin{thm}
Let $X$ be a separable topological space, and suppose that $(f_n)$ is an equicontinuous sequence of functions from $X$ into the  metric space $Y$, converging pointwise to  the function $f$ along the ultrafilter $\alpha\in\beta\omega\mysetminus \omega$.
If $\alpha$ is a P-point, then there exists $N\in\alpha$  such that
\begin{equation}\label{E1}
\lim_{\substack{n\to\infty\\ n\in N}} f_n(x)=f(x)\quad\forall x\in X.
\end{equation}
\end{thm}
\begin{proof}
If $U$ is an equicontinuous set of functions mapping the separable space $X$ into $Y$, then, the restriction to $U$ of the product topology on $Y^X$ is metrizable.  So the assertion follows by Proposition \ref{Ppoint1}.
\end{proof}

When the ultrafilter is selective, one can use deep Mathias's Theorem to prove that the same conclusion can be drawn when one relaxes equicontinuity  to measurability.  The proof of the next theorem follows the same arguments of \cite[Theorem 1.5]{FaPhSt2010}.  We prefer to present it in this more general setup because we believe that it could be interesting in its own right.

\begin{thm}\cite[Theorem 1.5]{FaPhSt2010}\label{t2}
Let $X$ be a Polish space, and suppose that $(f_n)$ is a sequence of Borel functions from $X$ into  a metric space $Y$, converging pointwise to the Borel function $f$ along the ultrafilter $\alpha$.  If $\alpha\in\beta\omega\mysetminus \omega$ is selective, then  there exists $N\in\alpha$ such that
\begin{equation}\label{e2}
\lim_{\substack{n\to\infty\\ n\in N}}f_n(x)=f(x)\qquad\forall x\in X.
\end{equation}
\end{thm}
\begin{proof}
Let $\rho$ denote the metric function on $Y$.  Consider the function
\begin{align*}
\Phi:X\times \N&\to \pp(\N)\\
(x,i)&\mapsto \left\{n\in\N:\rho\bigl(f_n(x),f(x)\bigr)\,>\,1/i\right\},
\end{align*}
and define
\[\ee:=\{N\in [\N]^\infty:\exists (x,i)\in X\times\N\ \st\ N\subset \Phi(x,i)\}.\]
If one where to show that $\ee$ is analytic, the result would follow by Mathias's Theorem as follows:   There is $N\in \alpha$
such that $[N]^\infty\subset \ee$ or $[N]^\infty\cap \ee=\emptyset$. Since we cannot have the first possibility because of our assumption
$\lim_{n\to\alpha} f_n(x)=f(x)$ for every $x\in X$, it follows that $[N]^\infty\cap \ee=\emptyset$; i.e. for every $(x,i)\in X\times \N$ the set
$\left\{n\in\N:\rho\bigl(f_n(x),f(x)\bigr)\,>\,1/i\right\}$
is finite.  So (\ref{e2}) would follow.

We show that $\Phi$ is a Borel function -- since $X\times \N$ is a Polish space, and  $[\N]^\infty$ (being a co-countable set in a metric space) is a
Borel subset of $\pp(\N)$,  this would imply that $\ee=[\N]^\infty\cap \Phi(X\times \N)$ is analytic.

For any pair $(F,G)\in [\N]^{<\omega}\times [\N]^{<\omega}$ define
\[\um(F,G):=\left\{N\in\pp(\N): N\cap F=\emptyset\text{ and }N\cap G=G\right\}.\]
It is clear that $\{\um(F,G):(F,G)\in [\N]^{<\omega}\times [\N]^{<\omega}\}$ is a countable base for $\pp(\N)$.  For our purpose, therefore, it suffices to show that
$\Phi^{-1}\bigl(\um(F,G)\bigr)$ is a Borel set.  This follows because
\[\Phi^{-1}\bigl(\um(F,G)\bigr)=\bigcup_i X_i\times\{i\},\]
where
\[X_i:=\bigcap_{n\in F}\bigr\{x\in X:\rho\bigl(f_n(x),f(x)\bigr)\le1/i\bigr\}\ \cap\ \bigcap_{n\in G}\bigl\{x\in X:\rho\bigl(f_n(x),f(x)\bigr)>1/i\bigr\}\]
and the later is a Borel set by the assumption on the Borel-measurability of the functions $f_n$ and $f$.
\end{proof}

\begin{cor}\label{c1}  Let $X$ be a dual Banach space that is separable w.r.t. the w*-topology.  For every $n\in\N$ let $T_n\in B(X)$ be a compact operator and suppose that
$\lim_{n\to\alpha}T_nx=0$ for every $x\in X$, where $\alpha$ is a selective ultrafilter on $\N$. Then, there exists $N\in\alpha$ s.t.
\[\lim_{\substack{n\to\infty\\ n\in N}}\Vert T_nx\Vert=0\qquad\text{for every $x\in X$}.\]
\end{cor}
\begin{proof}
The closed unit ball $X_1$, of $X$, is a Polish space when equipped with the w*-topology.  Since $T_n$ is compact, it follows that $T_n$ is w*-norm continuous.  Thus, the theorem implies that there exists $N\in \alpha$ s.t. $\lim_{\substack{n\to\infty\\ n\in N}}\Vert T_n x\Vert=0$ for every $x\in X$.
\end{proof}

\section{The commutant of $B(H)$ in its ultrapower.}\label{sec5}

Given an ultrafilter $\alpha\in\beta\omega\mysetminus \omega$, when is the commutant of $B(H)$ in its ultrapower w.r.t. $\alpha$ trivial?

\subsection{Non-trivial commutant -- Flat ultrafilters.}

\begin{thm}\cite[Theorem 4.1]{FaPhSt2010}\label{Farahflat}
  Let $(s_n)$ be a flatness scale for the ultrafilter $\alpha\in\beta\omega\mysetminus \omega$,  satisfying {\rm(i)-(iii)} of Proposition \ref{flatP1}.  Let $\{e_k:k\in\N\}$ be the set of pairwise orthogonal, one-dimensional projections onto the canonical basis of $\ell^2\ (=H)$.  For every $n\in\N$ let
  \[a_n:=\sum_{k=1}^{\infty}s_n(k)\,e_k\]
   and let $\mathbf{a}:=(a_n)$.  Then $\mathbf{a}$ is a non-trivial element of  $\ff_\alpha(B(H))$.
\end{thm}

The proof makes use of a  `stratification' property of $B(H)$ as described in the  following interesting lemma.  This lemma was originally proved in \cite[Lemma 3.1]{Fa2011} and generalised in \cite[Lemma 4.6]{FaPhSt2010}.  For every $g\in\nn^\uparrow$ and $k\in\N$, let
\[p^g(k):=\sum_{i=g^{k-1}(1)}^{g^k(1)-1}e_i,\]
 and let
 \[D(g):=\sum_{k=1}^\infty\  p^g(k) \,B(H)\, p^g(k).\]

\begin{lem}\label{Farahlem1} For every $a\in B(H)$ and $\delta>0$, there exist $g_0, g_1\in \nn^\uparrow$, $a^0\in D(g_0)$ and $a^1\in D(g_1)$ such that $a-a^0- a^1$ is compact, $\Vert a^i\Vert \le 2\Vert a\Vert$ for $i\in\{0,1\}$, and  $\Vert a-a^0-a^1\Vert<\delta/2$.
\end{lem}

 Lemma \ref{Farahlem1} implies that the span of $\bigcup_{g\in\nn^\uparrow} D(g)$ is norm dense  in $B(H)$.  It follows, therefore, that a sufficient (and  necessary) condition for a given $(a_n)$ in $\ccc_\alpha(B(H))$ to be in $\ff_\alpha(B(H))$ is that $\lim_\alpha \Vert a_n b-ba_n\Vert=0$ for every $b\in \bigcup_{g\in\nn^\uparrow} D(g)$.  So, to  proceed to the proof of Theorem \ref{Farahflat} we fix an arbitrary $g\in\nn^\uparrow$.  Observe that $y_n:=\sum_{k=1}^\infty\ s_n(k)  p^g(k)$ is in the centre  of $\mathcal D(g)$.  The flatness property of $\alpha$ implies that $\{n\in\N: \Vert a_n-y_n\Vert\leq \varepsilon\}\in\alpha$ for each $\varepsilon>0$. Moreover, since each $a_n$ is compact and of norm one,  the sequence $(a_n)$ is non-trivial, i.e. $(a_n)$ is a non-trivial element of  $\ff_\alpha(B(H))$.

We shall now proceed to prove the necessity of the flatness property for $\ff_\alpha(B(H))$ to be non-trivial.  This answers \cite[Question 5.2]{FaPhSt2010}.

\begin{rem}\label{rem1}
  Note that if $\ff_\alpha(B(H))$ is non-trivial,  then, in view of \cite[Corollary 2.21]{Ki2006}, and since the set of finite rank operators $F(H)$ is norm dense in $K(H)$,  there is  $\mathbf 0\neq \mathbf a\in \ff_\alpha(B(H))\cap \ccc_\alpha(F(H))$.  Moreover, we can  assume that $\mathbf a$ is of norm one and self-adjoint,  i.e. if $\mathbf a=(a_n)$, we can suppose that each $a_n$ has norm one, is of finite rank and is self-adjoint.  In addition, by making a small perturbation (if necessary) we can further assume that all the $a_n's$ are distinct, so we can identify $\alpha$ with an ultrafilter on $\{a_n: n\in\N\}$.
\end{rem}

\begin{thm}\label{answer}
Let $\alpha\in\beta\omega\mysetminus \omega$.  Then $\ff_\alpha(B(H))$ is non-trivial if and only if $\alpha$ is flat.
\end{thm}
\begin{proof}
In Theorem \ref{Farahflat} it was already shown that if $\alpha$ is flat, then $\ff_\alpha(B(H))$ is not trivial.  Here we show the converse.

Suppose that $\ff_\alpha(B(H))$ is not trivial.  In view of Remark \ref{rem1}, there exists $\mathbf{a}=(a_i)\in\ff_\alpha(B(H))$ such that every $a_i$ has norm one and  has finite rank.   Let $\{\xi_i: i\in\N\}$ be a fixed orthonormal basis of $H$.  For each $k\in\N$ let $q_k$ denote the projection of $H$ onto $\spn\{\xi_i:i< k\}$.  Note that $q_1=0$.
For every $n,m\in\N$ define $s_n(m):=\Vert a_n(\mathds 1-q_m)\Vert$.
 Observe that the sequence $(s_n)$ satisfies {\rm(i)} and {\rm(ii)} of Proposition \ref{flatP1}.  So, either $(s_n)$ is a flatness scale for $\alpha$ and we're done, or  there is $\varepsilon>0$ and a strictly increasing sequence $(m_k)$ in $\N$ such that
  \[A:=\bigcup_{k\in\N}\{n\in \N: s_n(m_k)-s_n(m_{k+1})>\varepsilon\}\]
  belongs to $\alpha$. Without loss of generality we can suppose that $(m_k)$ increases in a way  so that  the sequence $(d_k)$ where $d_k:=m_{k+1}-m_k$ is also strictly increasing.
  Observe that if $s_n(m_k)-s_n(m_{k+1})>\varepsilon$, then
  \[\Vert a_n (q_{m_{k+1}}-q_{m_{k}})\Vert=\Vert a_n(\mathds 1-q_{m_k})-a_n(\mathds 1-q_{m_{k+1}})\Vert\ge s_n(m_k)-s_n(m_{k+1})>\varepsilon.\]
  For every $k\in\N$ define $p_k:=q_{m_{k+1}}-q_{m_{k}}$.  Then $(p_k)$ is a sequence of pairwise orthogonal, finite-rank projections and for every $n\in A$ there exists $k\in\N$ such that $\Vert a_n p_k\Vert>\varepsilon$.  Observe that the dimension of the range space of $p_k$ equals $d_k$.
  For every $n,j\in\N$ define
  \[t_n(j):=\sup_{k\ge j} \Vert a_n p_k\Vert.\]
Note that $t_n(j+1)\le t_n(j)\le t_n(1)> \varepsilon$ for every $j\in \N$ and $n\in A$.
 We show that $(t_n)$ is flatness scale for $\alpha$.  Let $\delta>0$ be given and let $(\ell_i)$ be a strictly increasing sequence in $\N$.  Want to show that
  \[\{n\in\N:t_n(\ell_{i})-t_n(\ell_{i+1})\ge\delta\ \exists i\in\N\}\notin \alpha.\]
  It is harmless to assume that $\ell_{i+2}-\ell_{i+1}\ge\ell_{i+1}-\ell_i$ for every $i\in\N$.  Together with the fact that the sequence $(d_k)$ is strictly increasing, this assumption guarantees the existence of a partial isometry $u$ satisfying $u^\ast u=\mathds 1$ and
  \[u p_{\ell_i + k}\le p_{\ell_{i+1}+k}\,,\]
  for every $0\le k\le \ell_{i+1}-\ell_i-1$ and
  for every $i\in\N$.  For every $i\in\N$ and $n\in\N$ let $\mu(n,i)$ be the number satisfying $\ell_i\le \mu(n,i)<\ell_{i+1}$ and such that $\Vert a_n p_{\mu(n,i)}\Vert=\max\{\Vert a_n p_k\Vert : \ell_i\le k<\ell_{i+1}\}$.  Then
  \[
  t_n(\ell_i)-t_n(\ell_{i+1})=\begin{cases}
                                  0,  \mbox{if } \Vert a_n p_{\mu(n,i)}\Vert\le t_n(\ell_{i+1})\\
                                  \Vert a_n p_{\mu(n,i)}\Vert- t_n(\ell_{i+1}),  \mbox{otherwise},
                                \end{cases}\]
and therefore, if $i\in\N$ and $n\in\N$ are such that $t_n(\ell_i)-t_n(\ell_{i+1})>\delta$, then
\begin{align*}
\delta<t_n(\ell_i)-t_n(\ell_{i+1})&=\Vert a_n p_{\mu(n,i)}\Vert- t_n(\ell_{i+1})\\
                                &=\Vert u a_n p_{\mu(n,i)}\Vert - \sup_{k\ge\ell_{j+1}} \Vert a_n p_k\Vert\\
                                &\le \Vert u a_n p_{\mu(n,i)}\Vert - \Vert a_n u p_{\mu(n,i)}\Vert\\
                                &\le \Vert (u a_n-a_n u)p_{\mu(n,i)}\Vert\\
                                &\le \Vert ua_n-a_nu\Vert.
                                \end{align*}
                                This shows that
                                \[\{n\in\N:t_n(\ell_{i})-t_n(\ell_{i+1})\ge\delta\ \exists i\in\N\}\subset\{n\in\N:\Vert ua_n-a_n u\Vert\ge \delta\},\]
                                and completes the proof
                                \end{proof}

\subsection{Trivial commutant -- $3$f-property}

Although the following theorem follows also by Theorem \ref{answer} and Corollary \ref{f3f}, we here give a direct proof.

\begin{thm}\label{t1}
If $\alpha\in\beta\omega\mysetminus \omega$ has the $3$f-property, then  $\ff_\alpha(B(H))$ is trivial.
\end{thm}
\begin{proof}
Suppose that $\ff_\alpha(B(H))$ is not trivial and that $\alpha$ has the $3$f-property.  We shall seek a contradiction.  In view of Remark \ref{rem1}, there exists $\mathbf{a}=(a_i)\in\ff_\alpha(B(H))$ such that every $a_i$ has norm one and has finite rank.  Pick an arbitrary strictly decreasing sequence $(p_j)$ of projections on $H$ converging to $0$ w.r.t. the strong operator topology and such that $p_1=\mathds 1$.  For every $(i,j)\in\N^2$ let $a_{i\,j}:=\Vert a_ip_j\Vert$.  Observe that the infinite matrix $(a_{i\,j})$ satisfies the hypothesis of Proposition \ref{3fP1}, and therefore, by the same proposition and by Remark \ref{rem2}, for every $0<\delta<1/4$ there exists $A\in\alpha$ and a strictly increasing sequence $(j_k)$ such that
  \[\bigcup_{k\in\N}\{i\in A:a_{i\,j_k}-a_{i\,j_{k+1}}>\delta\}\in\alpha.\]
Let $u$ be a partial isometry such that $uu^\ast=\mathds 1$ and  $u p_{j_k}\le p_{j_{k+1}}$ for every $k\in\N$.  Then
\begin{align*}
\Vert u\, a_i\,-\,a_i\,u\Vert\,\ge\,\Vert u\,a_i\,p_{j_k}\,-\,a_i\,u\,p_{j_k}\Vert\,&\ge\, \Vert u\,a_i\,p_{j_k}\Vert\,-\,\Vert a_i\,u\,p_{j_k}\Vert\\
&\ge\,\Vert a_i\,p_{j_k}\Vert\,-\,\Vert a_i\,p_{j_{k+1}}\Vert\\
&=\,a_{i\,j_k}\,-\,a_{i_{j_{k+1}}}.
\end{align*}
This implies that $\{i\in\N: \Vert u\,a_i\,-\,a_i\,u\Vert>\delta\}\in\alpha$, contradicting the centrality assumption of $\mathbf{a}$.
\end{proof}

 \begin{rem}
    In \cite[Theorem 2]{FaPhSt2010},  it was shown that when $\alpha$ is selective, $\ff_\alpha(B(H))$ is trivial.  This follows by Theorem \ref{t2}  and by invoking a result by Sherman \cite{She2010} which says that no factor von Neumann algebra (in particular $B(H)$) can admit any non-trivial central sequences\footnote{Recall that a sequence $(a_n)$ in a von Neumann algebra $M$ is a non-trivial  central sequence if $\lim_n\Vert a_na-a a_n\Vert=0$ for every $a\in M$ and $\lim_n\inf_{\lambda\in\C}\Vert a_n-\lambda\mathds 1\Vert\neq 0$ for every $\lambda\in\C$.}. Indeed, if $(a_n)$ is a representing sequence of $\mathbf a\in \ff_\alpha(B(H))$ and each $a_n$ is a compact operator, the sequence of operators $T_n:B(H)\to B(H)$ defined by $x\mapsto a_nx-xa_n$ satisfy the hypothesis of Corollary \ref{c1}, and therefore there exists an increasing sequence $(n_i)$ in $\N$ such that $\{n_i:i\in\N\}\in \alpha$ and $\lim_{i\to\infty}T_{n_i}x=0$ for every $x\in B(H)$.  Note that besides from being a proper generalization, Theorem \ref{t1}, allows for a more elementary proof.
  \end{rem}

\subsection{Trivial commutant -- quasi P-point}

\begin{thm}\label{qp}
  If $\alpha\in\beta\omega\mysetminus \omega$ is a quasi P-point, then $\ff_\alpha(B(H))$ is trivial.
\end{thm}
\begin{proof}
Suppose that $\ff_\alpha(B(H))$ is not trivial and that $\alpha$ is a quasi P-point.  We shall seek a contradiction.  In view of Remark \ref{rem1}, there exists\footnote{To avoid cumbersome notation with subscripts we write $(a(i))$ instead of $(a_i)$.}  $\mathbf{a}=(a(i))\in\ff_\alpha(B(H))$ such that every $a(i)$ has norm one, has finite rank and is self-adjoint.  Moreover, we can  assume that all the $a(n)$'s are distinct, so we can identify $\N$ with $\{a(i):i\in\N\}$ by identifying $i$ with $a(i)$.

We first note that $\alpha$ cannot be Cauchy because if otherwise, it would converge in $K(H)$ (w.r.t. the norm)  and the limit -- being in the centre of $K(H)$ -- must be equal to $0$,  contradicting the assumption that $(a(n))$ is not trivial.
Therefore, Proposition \ref{P3} implies that there is  $\delta>0$ such that every $A\in\alpha$ contains an infinite $\delta$-discrete subset.

Let $p(n)$ be the range projection of $a(n)$ and let $q(n):=\bigvee_{m\le n}p(m)$.   For every $n\in\N$  the set
\[Y(n):=\left\{m\ge n\,:\, \left\Vert q(n)\,a(m)\,\bigl (q(m)-q(n)\bigr)\right\Vert<\delta/12\right\}\]
 belongs to $\alpha$ by the assumption of centrality.  Let \[
X'(n):= \left\{m\in Y(n)\,:\,\left\Vert (q(m)-q(n))\, a(m)\, (q(m)-q(n))\right\Vert<\delta/12\right\}
\]
For any $s,t$ in $X'(n)$ we note that
\begin{equation}
\begin{split}\label{est}
  \Vert a(s)-a(t)\Vert\,&\le\, \bigl\Vert q(n)\, a(s)\, q(n)\,-\,q(n)\, a(t)\, q(n)\bigr\Vert\\
  &+\, 2\bigl\Vert (q(s)-q(n))\, a(s)\, q(n)\bigr\Vert\,+\,2\bigl\Vert (q(t)-q(n))\,a(t)\, q(n)\bigr\Vert\\
  &+\,\bigl\Vert (q(s)-q(n))\,a(n)\,(q(s)-q(n))\bigr\Vert\\
  &+\,\bigl\Vert (q(t)-q(n))\,a(t)\,(q(t)-q(n))\bigr\Vert\\
  &<\,\Vert q(n)\,a(s)\, q(n)\,-\,q(n)\,a(t)\,q(n)\Vert\,+\,\delta/2.
  \end{split}
  \end{equation}

Let $X(n):=\N\mysetminus X'(\N)$.  We shall need to consider two cases:

\noindent\emph{Case {\rm(i)}}  Suppose that there exists $A\in\alpha$ and $\delta'>0$ such that $A$ is $\delta'$-discrete.  By taking a smaller $\delta$ (if necessary) we can suppose that $A$ is $\delta$-discrete.     The  above estimation shows that  if $X'(n)\cap A$ is infinite for some $n\in\N$, the unit ball of $q(n)\,B(H)\,q(n)$ would contradictorily contain an infinite $\delta/2$-discrete subset.  Thus, $A\cap X(n)\in\mathcal F(A)$ for every $n\in\N$.  Observe that $X(n)\cap A\downarrow\emptyset$ in $\alpha$ and the indicator function $f$ associated with the sequence $(X(n)\cap A:n\in\N)$ is $\alpha$-finite-to-one.

\noindent \emph{Case {\rm(ii)}}  Suppose that $\alpha$ is neither discrete nor Cauchy, i.e. by Proposition \ref{quasi} we have
\[\forall f\in\nn\qquad\forall\varepsilon>0\quad\exists\,\alpha\text{-faithful}\ g\in\nn\ \st\  g\prec_\varepsilon f.\]
Note that the estimation of (\ref{est}) implies that if $X'(n)\in\alpha$ for some $n\in\N$ then --  by Proposition \ref{P3} -- the unit ball of $q(n)\,B(H)\,q(n)$ would contradictorily contain an infinite $\delta/2$-discrete subset.  Therefore $X(n)\in\alpha$ for every $n\in\N$.  Let $f$ denote the indicator function associated with  the sequence $(X(n):n\in\N)$.

In both cases, therefore, it is possible to find an $\alpha$-injective $g\in\nn$  satisfying $g\prec_\varepsilon f$, where $\varepsilon<{\delta}/{24}$.  Let $h\in\nn^\uparrow$ satisfy $h(k)>\min g^{-1}\{k\}$ for every $k\in\range g$.  By Lemma \ref{l1}, there exists a consecutive\footnote{i.e. $i<j$ for every $i\in I_n$, $j\in I_m$ and $n<m$.} family of finite sets $\{I_k:k\in\N\}$ of $\N$ such that $\bigcup_k I_k\in \alpha^g$ and $g^{-1} I_{k+1}\subset_\varepsilon X({h(\max I_k)})$.  Without loss of generality, we can assume that $I_k\subset \range g$.   For every $k\in\N$, let $I_k=\bigl\{i^k(1),\,i^k(2),\,\dots,\,i^k(l_k)\bigr\}$ and for $1\le p\le l_k$ let $n^k(p):=\min g^{-1}\{i^k(p)\}$.   Observe that $n^k(p)\le h(\max I_k)$ for every $1\le p\le l_k$.\\

 Let $e_1:=q\bigl(h(\max I_1)\bigr)$ and for $k>1$ let $e_k:=q\bigl(h(\max I_k)\bigr)\,-\,q\bigl(h(\max I_{k-1})\bigr)$.  Then $(e_k)$ is a sequence of pairwise orthogonal projections satisfying:
\begin{itemize}
\item $\left(\sum_{i=1}^{k}e_i\right)\,a\bigl(n^k(p)\bigr)\,=\,q\bigl(h(\max I_k)\bigr)\,a\bigl(n^k(p)\bigr) \,=\,a\bigl(n^k(p)\bigr),$ \\
\item $ a\bigl(n^k(p)\bigr)\,\sum_{i>k}e_{i}\,=\,0$, and\\
\item $\bigl\Vert\,e_k\,a\bigl(n^k(p)\bigr)\, e_k\,\bigr\Vert\,>\,  \frac{\delta}{12}-\varepsilon\,>\,\frac{\delta}{24}$,
\end{itemize}
for every $k> 1$ and $1\le p\le l_k$.  Let $\eta^k(p)$ be a unit vector in $e_kH$ such that
\[\bigl\Vert \, a\bigl( n^k(p)\bigr)\,(\eta^k(p))\bigr\Vert>\delta/24\qquad(1\le p\le l_k).\]

Let $(f_k)$ be a sequence of pairwise orthogonal projections in $B(H)$ such that $f_k\le \sum_{i>k} e_i$ and $e_k\sim f_k$ for every $k\in\N$.  Let $u_k$ be a partial isometry on $H$ satisfying $u_k^\ast u_k=e_k$ and $u_k u_k^\ast=f_k$. Then, $u:=\sum_ku_k$ belongs to $B(H)$ and satisfies
    \[a\bigl(n^k(p)\bigr) u\,(\eta^k(p))\,=\,a\bigl(n^k(p)\bigr)u_k\,(\eta^k(p))\,=\,0,\]
    and
    \[\bigl\Vert u\,a\bigl(n^k(p)\bigr)\,(\eta^k(p))\,\bigr\Vert\,=\,\left\Vert\sum_{i=1}^k(u_ie_i)\,a\bigl(n^k(p)\bigr)\,(\eta^k(p))\right\Vert\,=\,\bigl\Vert a\bigl(n^k(p)\bigr)\,(\eta^k(p))\,\bigr\Vert\,>\,\delta/24,\]
    for every $k\in\N$ and $1\le p\le l_k$.  This shows that $\bigl\Vert a\bigl(n^k(p)\bigr)\, u\,-\,u\,a\bigl(n^k(p)\Vert>\delta/24$ for every $k\in\N$ and $1\le p\le l_k$. \\

So, we have
    \[ \bigcup_k\bigl\{i^k(1),\,i^k(2),\,\dots,\,i^k(l_k)\bigr\}\,\subset\, \bigl\{g(n)\,:\,n\in\N,\ \Vert a(n)\,u\,-\,u\,a(n)\Vert>\delta/24\bigr\},\]
i.e. $\{g(n):n\in \N,\ \Vert a(n)\,u\,-\,u\,a(n)\Vert>\delta/24\}$ belongs to $\alpha^g$.  Since $g$ is $\alpha$-injective, this implies that $\{n\in\N\,:\,\Vert a(n)\,u\,-\,u\,a(n)\Vert\le\delta/24\}\in\alpha$.  But, since $\mathbf a\in\ff_\alpha(\aaa)\cap \ccc_\alpha(F(H))$, we also have that
$\{n\in\N\,:\,\Vert a(n)\,u\,-\,u\,a(n)\Vert\le\delta/24\}\in \alpha$.
 The contradiction completes the proof.
\end{proof}

Question {\rm(i)} below becomes more relevant in the light of Theorem \ref{answer}.  We believe  that Question {\rm(ii)} is also natural to ask.

\begin{prob}
  \begin{enumerate}[{\rm(i)}]
    \item Do non-flat ultrafilters exist in ZFC?
    \item What is the precise relation between the notion of quasi P-points and flat ultrafilters?
  \end{enumerate}
\end{prob}

\end{document}